%% file: Holomorphic_curves_in_Shimura_varieties.tex
\pdfoutput=1

\documentclass{scrartcl}
	\usepackage{amsmath, amsfonts, amssymb, amsthm, graphicx, mathrsfs, mathtools, bm}

\usepackage[utf8]{inputenc}
\usepackage[T1]{fontenc}
\usepackage{lmodern}
\usepackage[english]{babel}	
\usepackage{graphicx}
\graphicspath{{./Immagini/}}

\usepackage[font={small,it}, labelfont=bf]{caption} 
\usepackage{subcaption}
\usepackage{comment}
\usepackage{xspace}

\usepackage[pdftex, colorlinks=true, pdfstartview=FitV, linkcolor=blue, citecolor=blue, urlcolor=blue,p lainpages=false,bookmarks,final]{hyperref}
\hypersetup{
  bookmarksopen, 
  bookmarksopenlevel=2, 
  pdftitle={Holomorhic curves in Shimura varieties},
  pdfauthor={Michele Giacomini},
  }

\usepackage{xfrac}
\usepackage[usenames,dvipsnames]{color}
\usepackage{enumerate}
\usepackage{csquotes}
\usepackage[
style = alphabetic,%
sorting=nyt,maxnames=5,maxalphanames=5,%
firstinits=true,    
isbn=false, doi=false, url=false
]{biblatex}

\AtEveryBibitem{%
\clearfield{note}
\clearfield{edition}
\clearlist{location}
}

\usepackage{tikz}
\usetikzlibrary{bending}
\usetikzlibrary{arrows.meta}

\title{Holomorphic curves in Shimura varieties}
\author{Michele~Giacomini}

\date{}

\newcommand{\R}{\mathbb{R}}

\newcommand{\Z}{\mathbb{Z}}

\newcommand{\Q}{\mathbb{Q}}
\newcommand{\N}{\mathbb{N}}

\newcommand{\End}{\mathrm{End}}

\newcommand{\de}{\partial}

\newcommand{\abs}[1]{\left\vert #1\right\vert}

\renewcommand{\bar}[1]{\overline{#1}}

\DeclareMathOperator{\im}{Im}
\DeclareMathOperator{\re}{Re}

\DeclareMathOperator{\alg}{alg}

\DeclareMathOperator{\Vol}{Vol}
\DeclareMathOperator{\Zar}{Zar}

\newtheoremstyle{break}
  {\topsep}{\topsep}%
  {\itshape}{}%
  {\bfseries}{}%
  {\newline}{}%

\theoremstyle{definition}
\newtheorem{definition}{Definition}[section]
\theoremstyle{plain}
\newtheorem{theorem}[definition]{Theorem}
\newtheorem*{theorem*}{Theorem}   

\newtheorem{lemma}[definition]{Lemma}
\newtheorem{proposition}[definition]{Proposition}

\theoremstyle{remark}

\newtheorem{example}[definition]{Example}

\theoremstyle{break}

\numberwithin{equation}{section}

\DefineBibliographyStrings{english}{%
  bibliography = {References},
}

\begin{document}

\maketitle

\begin{abstract}
We prove an hyperbolic analogue of the Bloch-Ochiai theorem about the Zariski closure of holomorphic curves in abelian varieties. We consider the case of non compact Shimura varieties completing the proof of the result for all Shimura varieties. The statement analysed was first formulated and proven by Ullmo and Yafaev for compact Shimura varieties.
\end{abstract}

\tableofcontents


\input{introduction}

\input{notation}

\input{intersections-of-holomorphic-curves-and-fundamental-domains}

\input{proof}

\printbibliography[heading=bibintoc]

\end{document}

%% file: introduction.tex
\section{Introduction} 
\label{sec:introduction}

The Bloch-Ochiai theorem \autocite[cf.][Chapter 9, 3.9.19]{Kobayashi98} states that the Zariski closure of an holomorphic curve in an abelian variety is a coset of an abelian subvariety.

\begin{theorem}[Bloch-Ochiai]
	\label{thm:Bloch-Ochiai}
	Let $A$ be an abelian variety and $f:\mathbb{C}\to A$ a non-constant holomorphic map. Then the Zariski closure of $f(\mathbb{C})$ is a translate of an abelian subvariety.
\end{theorem}

In \cite{UllmoYafaev16}, Ullmo and Yafaev formulate and prove an analogue of this result for compact Shimura varieties.

Let $\mathcal{D}$ be a hermitian symmetric space realised as a bounded symmetric domain in $\mathbb{C}^{n}$ via the Harish-Chandra embedding\footcite[see][Chapter 4]{Mok89}, $G$ its isometry group and $\Gamma\subset G(\mathbb{R})$ an arithmetic lattice. Let $S=\Gamma\backslash\mathcal{D}$. Assume that $S$ is a component of a Shimura variety; in particular $G$ is defined over $\mathbb{Q}$ and $\Gamma$ is a congruence subgroup of $G(\mathbb{Q})$. Finally consider a holomorphic function $f:\mathbb{C}^{m}\to \mathbb{C}^{n}$ such that $f(\mathbb{C}^{m})\cap \mathcal{D}\neq \emptyset$. 

\begin{theorem}[{\cite[Theorem 1.2]{UllmoYafaev16}}]
	\label{thm:holomorphic curves in compact Shimura}
	Let $\pi:\mathcal{D}\to S$ be the quotient map, $f$ as above and $V=f(\mathbb{C}^{m})\cap \mathcal{D}$. Assume $S$ is compact then the components of the Zariski closure $\Zar(\pi(V))$ of $\pi(V)$ in $S$ are weakly special subvarieties of  $S$.
\end{theorem}

For general definitions about Shimura varieties and weakly special subvarieties see \cite{UllmoYafaev14} and the references contained there.

Along with the Bloch-Ochiai theorem the above result draws inspiration from the hyperbolic Ax-Lindemann theorem, first proven by \citeauthor{UllmoYafaev14} in \cite{UllmoYafaev14} for compact Shimura varieties and then in general by \citeauthor{KlinglerUllmoYafaev16} in \cite{KlinglerUllmoYafaev16}.

\begin{theorem}[Ax-Lindemann]
	\label{thm:Ax-Lindemann}
	Let $Y\subset \mathcal{D}$ be an algebraic subset of $\mathcal{D}$\footnote{An algebraic subset of $\mathcal{D}$ is a component of the intersection of an algebraic subset of $\mathbb{C}^{n}$ with $\mathcal{D}$}. Then the components of the Zariski closure $\Zar(\pi(Y))$ are weakly special.
\end{theorem}

Our aim is to prove the result analogous to theorem \ref{thm:holomorphic curves in compact Shimura} for all Shimura varieties (not necessariòy compact), thus completing the proof of:

\begin{theorem}[Main Result]
	\label{thm:main result}
	Let $\pi:\mathcal{D}\to S$ be the quotient map, $f$ as above and $V=f(\mathbb{C}^{m})\cap \mathcal{D}$. Then the components of the Zariski closure $\Zar(\pi(V))$ of $\pi(V)$ in $S$ are weakly special subvarieties of  $S$.
\end{theorem}

As in \cite{UllmoYafaev16}, the proof follows the general lines of the proof of the hyperbolic Ax-Lindemann theorem. In particular, it relies on the theory of o-minimal structures and specifically on the use of Pila-Wilkie's theorem on counting rational points in definable sets (see theorem \ref{thm:Pila-Wilkie}). 

The main steps of the proof are as follows. First we reduce to proving the result separately on several 'branches' $V_{i}$ of the portion of the image of $f$ contained in $\mathcal{D}$ in such a way that each $V_{i}$ is definable in $\mathbb{R}_{an,exp}$. Then, we use toroidal compactifcations of Shimura varieties, Pila-Wilkie's and the Ax-Lindemann-Weierstrass theorem to prove that the Zariski closure of the image of $U_{i}$ contains a Zariski dense set of weakly special subvarieties. Here, the crucial part is Lemma \ref{lemma:lemma 5.8}, which asserts that the volume of the intersection between one of these definable curves $U_i$ in $\mathcal{D}$ and a translate $\gamma \mathscr{F}$ of a fixed fundamental domain $\mathscr{F}$ for the action of $\Gamma$ on $\mathcal{D}$ is bounded independently of $\gamma\in \Gamma$. Finally, we conclude the proof of the main result of the paper using a result of Ullmo \cite[Théorème 1.3]{Ullmo14} and induction on the dimension.

We point out that, although our result is independent of the realisation of the symmetric domain $\mathcal{D}$ uniformising $S$\footnote{see \autocite{Ullmo14} for the definition of realisation of a symmetric domain.}, we use in a crucial way that there is a bounded realisation. Indeed this allows us to reduce the proof to the definable sets $V_{i}$ and is again used in a fundamental way in the proof of the above cited Lemma \ref{lemma:lemma 5.8}.

To stress further the importance of the boundedness of $\mathcal{D}$, we point out that questions related to the Bloch-Ochiai theorem in the abelian setting were investigated using o-minimal techniques by \citeauthor{UllmoYafaev17} in \cite{UllmoYafaev17}. In this setting the authors were not able to prove the full Bloch-Ochiai theorem with the present techniques; this is ultimately due to the fact that the symmetric space uniformizing an abelian variety of dimension $d$ is $\mathbb{C}^{d}$ which has no bounded realisation.

\paragraph{Acknowledgements} 
\label{par:acknowledgments}
I would like to thank my supervisor Andrei Yafaev for pointing me to this problem and some helpful discussions. I would also like to thank Jacob Tsimerman who was the first to point out that the result in Lemma \ref{lemma:lemma 5.8} could be generalised from the cocompact case in \cite{UllmoYafaev16} to the general case using the existence of a bounded realisation for $\mathcal{D}$.

\noindent This work was supported by the Engineering and Physical Sciences Research Council [EP/L015234/1]. 
The EPSRC Centre for Doctoral Training in Geometry and Number Theory (The London School of Geometry and Number Theory), University College London.

%% file: notation.tex
\section{Preliminaries} 
\label{sec:preliminaries}

First we fix some notation.

\begin{itemize}
	\item Let $(G, X)$ be a Shimura datum and let $\mathcal{D}$ be a connected component of $X$. $\mathcal{D}$ is an hermitian symmetric domain which we realise as a bounded hermitian symmetric domain in the holomorphic tangent space $\mathfrak{p}\cong \mathbb{C}^{n}$ at a point $x\in \mathcal{D}$ via the Harish-Chandra embedding.
	\item Let $G(\mathbb{Q})^{+}$ be the stabiliser of $\mathcal{D}$ in $\mathbb{G}(\mathbb{Q})$ and $\Gamma \subset G(\mathbb{Q})^{+}$ a neat arithmetic subgroup; we may assume there is a faithful finite dimensional representation $\rho:G\to GL(E)$ defined over $\mathbb{Q}$ and a lattice $E_{\mathbb{Z}}\subset E$ such that $\Gamma=G(\mathbb{Z})=G(\mathbb{Q})\cap GL(E_{\mathbb{Z}})$.
	\item Let $\Sigma\subset \mathcal{D}$ be a Siegel set for the action of $\Gamma$ such that there exists a finite set  $J\subset G(\mathbb{Q})$ such that $J. \Sigma = \mathscr{F}$ is a fundamental set for the action of $\Gamma$.
\end{itemize}

Let $f:\mathbb{C}^{m}\to \mathbb{C}^{n}$ be an holomorphic map such that $f(\mathbb{C}^{m})\cap \mathcal{D}\neq \emptyset$. We decompose 
\begin{equation}
	f^{-1}(f(\mathbb{C}^{m})\cap \mathcal{D})=\coprod_{i\in I} U_{i}
\end{equation}
as a disjoint union of connected components. By definition of $U_{i}$, 
\begin{equation}
	\bar{f(U_{i})}\cap \partial X\neq \emptyset;
\end{equation}
hence there exists $R_{i}$ such that $\bar{f(U_{i}\cap B(0, R_{i}))}\cap \partial X\neq \emptyset$.

By analytic continuation, it follows that the Zariski closures $Zar(\pi\circ f(U_{i}))$ and $Zar(\pi\circ f(U_{i}\cap B(0,R_{i})))$ are equal. Let $V_{i}=f(U_{i}\cap B(0,R_{i}))$ and $W_{i}=Zar(\pi(V))\subset S$. Following \cite{UllmoYafaev16}, we will deduce our main result \ref{thm:main result} from the following theorem.

\begin{theorem}
	\label{thm:thm 2.3}
	There exists a positive dimensional semialgebraic set $X$ in $G(\mathbb{R})$ such that
	\begin{equation}
		X\cdot f(B(0,R_{i})\cap U_{i}) \subset \pi^{-1}(W_{i}).
	\end{equation}
\end{theorem}

Following \cite{UllmoYafaev16}, we now briefly describe how the main result follows from the above theorem. Let $P\in V_{i}$, let $X\subset G(\mathbb{R})$ be a maximal semialgebraic subset such that $X\cdot P\subset \pi^{-1}(W_{i})$. By the above theorem this has positive dimension and by \cite[Lemma 4.1]{PilaTsimerman13}, it is a complex algebraic subset. By the Ax-Lindemann theorem \ref{thm:Ax-Lindemann}, the Zariski closure $Zar(\pi(X\cdot P))\subset W_{i}$ is a union of weakly special subvariety. Hence for each point of $P\in \pi(V_{i})$ there is a weakly special subvariety $Y$ such that $P\in Y\subset W_{i}$. This proves the following

\begin{theorem}
 	\label{thm:zar dense ws subvars}
 	$W_{i}=Zar(\pi(V_{i}))$ contains a Zariski dense subset of weakly special subvarieties.
 \end{theorem} 

Now we proceed by induction on the dimension of $W_{i}$ to show that theorem \ref{thm:zar dense ws subvars} implies the main result. The case of dimension zero is trivial since all points are weakly special. If $W_{i}$ is special we are done, otherwise by \cite[Théorème 1.3]{Ullmo14}, it follows that the smallest special subvariety $S'\subset S$ containing $W_{i}$ can be decomposed as a product $S'=S_{1}\times S_{2}$  with both factors non trivial, such that
\begin{equation}
	W_{i}=S_{1}\times V'
\end{equation}
for some subvariety $V'\subset S_{2}$.

Let $(G',X')$ be the sub-Shimura datum of $(G,X)$ associated to $S'$. The above decomposition induces a decomposition of the adjoint datum $(G'^{ad},X'^{ad})$as a product $(G_{1},X_{1})\times (G_{2},X_{2})$ such that, for $i=1,2$, $S_{i}=\Gamma_{i}\backslash \mathcal{D}_{i}$ for some suitable arithmetic subgroup $\Gamma_{i}$ of $G_{i}(\mathbb{Q})^{+}$. We can realise both $\mathcal{D}_{i}$ as bounded symmetric domains inside their holomorphic tangent spaces $\mathfrak{p}_{i}$. Then we can write $f:\mathbb{C}^{m}\to \mathfrak{p}_{1}\times \mathfrak{p}_{2}\subset \mathfrak{p}$ as $f=(f_{1},f_{2})$. It now follows that $V'$ is exactly the Zariski closure of $\pi\circ f_{2}(U_{i})$. By theorem \ref{thm:zar dense ws subvars}, $V'$ contains a Zariski dense set of weakly special subvarieties and by the inductive hypothesis it is weakly special.

We now recall two results about the structure of $\mathcal{D}$ at the boundary we will need later.

\begin{proposition}[{\cite[Chapter III.4]{AshMumfordRapoportTai10}}]
	\label{prop:boundary comp decomposition}

Given a boundary component $F\subset \bar{\mathcal{D}}$, its normaliser $N(F)$ in $G$ is a parabolic subgroup and can be decomposed as follows
\begin{equation*}
	N(F)=(G_{h}(F)G_{l}(F)M(F))(V(F)U(F)),
\end{equation*}
where
\begin{itemize}
	\item $R(F)=(G_{h}(F)G_{l}(F)M(F))$ is a Levi factor of $N(F)$ and the product is direct modulo a finite central group
	\item $W(F)=(V(F)U(F))$ is the unipotent radical of $N(F)$
	\item $U(F)$ is the center of $W(F)$ and is a real vector space
	\item $V(F)=W(F)/U(F)$ is also a real vector space of even dimension $2l$
	\item $G_{h}(F)$ modulo a finite center is $Aut^{0}(F)$, all the other factor act trivially
	\item $G_{l}(F)$ modulo a finite center acts on $U(F)$ by inner automorphisms, the other factors commute with $U(F)$
	\item $M(F)$ is compact
\end{itemize}

\end{proposition}

\begin{proposition}[{\cite[Proposition 3.2]{KlinglerUllmoYafaev16}}]
	\label{prop:proposition 3.2}
	Fix a boundary component $F\subset \bar{\mathcal{D}}$. Define
	\begin{equation*}
		\mathcal{D}_{F}=\bigcup_{g\in U(F)_{\mathbb{C}}} g.\mathcal{D}\subset \mathcal{D}^{c}.
	\end{equation*}
	There is an holomorphic isomorphism $j: \mathcal{D}_{F}\to U(F)_{\mathbb{C}}\times \mathbb{C}^{l}\times F$. This isomorphism realises $\mathcal{D}$ as a Siegel domain of the third kind
	\begin{equation*}
		\mathcal{D}\mathrel{\mathop{\simeq}^{j}} \left\{ (x,y,t)\in U(F)_{\mathbb{C}}\times \mathbb{C}^{l}\times F \mathrel{| \im(x)+l_{t}(y,y)\in C(F)} \right\}
	\end{equation*}
	where $C(F)$ is a self-adjoint convex cone in $U(F)$ homogeneous under $G_{l}(F)$ and $l_{t}:\mathbb{C}^{l}\times \mathbb{C}^{l}\to U(F)$ is a symmetric bilinear form varying real-analytically with $t\in F$.

	Let $\Sigma\subset \mathcal{D}$ be a Siegel set for the action of $\Gamma$, as above. Then $\Sigma$ is covered by a finite number of open subsets $\Theta$ having the following properties. For each $\Theta$ there is a cone $\sigma$ with $\sigma\subset \bar{C(F)}$, a point $a\in C(F)$, realtively compact subsets $U'$, $Y'$ and $F'$ of $U(F)$, $\mathbb{C}^{l}$ and $F$ respectively such that the set $\Theta$ is of the form
	\begin{equation*}
	\begin{split}
		\Theta \mathrel{\mathop{\simeq}^{j}} &\left\{ (x,y,t)\in U(F)_{\mathbb{C}}\times \mathbb{C}^{l}\times F \mathrel{|} \re(x)\in U', \, y\in Y',\, t\in F' \text{ and }\right. \\
		&\left. \im(x)+l_{t}(y,y)\in \sigma + a\right\}.
	\end{split}
	\end{equation*}
\end{proposition}


%% file: intersections-of-holomorphic-curves-and-fundamental-domains.tex
\section{Holomorphic curves and fundamental domains} 
\label{sec:holomorphic_curves_and_fundamental_domains}

In this section we restrict attention to holomorphic curves; hence, we set $m=1$ and consider holomorphic maps $f:\mathbb{C}\to \mathbb{C}^{n}=\mathfrak{p}$ be a holomorphic map. Maintaining the same notation as in the last section, we fix some $i\in I$ and let $C=f(U_{i}\cap B(0,R_{i}))$.

\begin{definition}
	\label{def:NCT}
	Recall we fixed a faithful finite dimensional representation $\rho:G\to GL(E)$; for any $\gamma\in \Gamma$ write $\rho(\gamma)=(\gamma_{i,j})_{i,j}$. For any $\phi\in \End(E_{\mathbb{R}})$ define
	\begin{equation}
		\abs{\phi}_{\infty}=\max_{i,j}\abs{\phi_{i,j}}.
	\end{equation}
	Moreover, define the \emph{height} of $\gamma\in \Gamma$ as 
	\begin{equation*}
		H(\gamma)=\max(1,\abs{\gamma}_{\infty}).
	\end{equation*}
	Finally, let $T>0$, define
	\begin{equation*}
		N_{C}(T)=\# \left\{ \gamma\in \Gamma \mathrel{|} \gamma \mathscr{F}\cap C\neq\emptyset \text{ and } H(\gamma)\leq T\right\}
	\end{equation*}
\end{definition}

The aim of this section is to prove the following result.

\begin{theorem}
	\label{thm:intersection holomorphic curves and fundamental domains}
	There exist constants $c_{1},c_{2}>0$ such that for all $T>0$ sufficiently large
	\begin{equation*}
		N_{C}(T)\geq c_{1} T^{c_{2}}.
	\end{equation*}
\end{theorem}

In the proof of theorem \ref{thm:intersection holomorphic curves and fundamental domains} for algebraic curves given in \cite{KlinglerUllmoYafaev16}, the only part that the relies on the curve being algebraic is the analogue of the following lemma.

\begin{lemma}
	\label{lemma:lemma 5.8}
	There exists a positive constant $c_{3}$ such that for any $\gamma \in \Gamma$
	\begin{equation}
		Vol_{\gamma C}(\gamma C\cap \mathscr{F})\leq c_{3},
	\end{equation}
	where $\Vol_{\gamma C}$ is the volume with respect to the riemaniann metric on $\gamma C$ induced by the metric on $\mathcal{D}$.
\end{lemma}
\begin{proof}
	By definition, $\mathscr{F}= J. \Sigma$ for a finite subset $J\subset G(\mathbb{Q})$ and some Siegel subset $\Sigma\subset \mathcal{D}$. Hence it is sufficient to prove the theorem for the Siegel set $\Sigma$. In turn, every Siegel set is covered by a finite number of open subsets $\Theta$ as in proposition \ref{prop:proposition 3.2}, so it is sufficient to prove that
	\begin{equation}
		\Vol_{\gamma C}(\gamma C\cap \Theta)\leq c_{4}
	\end{equation}
	for some constant $c_{4}>0$.

	Let $\omega$ be the natural Kähler form on $\mathcal{D}$, then
	\begin{equation}
		Vol_{\gamma C}(\gamma C \cap \Theta)= \int_{\gamma C \cap \Theta} \omega
	\end{equation}
	On $\mathcal{D}_{F}$ we have the Poincaré metric defined by
	\begin{equation}
		\omega_{F}=\sum \frac{dx_{1}\wedge d\bar{x_{i}}}{\im(x_{i})^{2}}+\sum dy_{j}\wedge d\bar{y_{j}}+\sum df_{k}\wedge d\bar{f_{k}}.
	\end{equation}
	By a result of Mumford \cite[Theorem 3.1]{Mumford77}, there is a constant $c_{5}$ such that on $\mathcal{D}$
	\begin{equation}
		\omega\leq c_{5} \omega_{F}.
	\end{equation}
	Hence
	\begin{equation}
		Vol_{\gamma C}(\gamma C \cap \Theta)\leq \int_{\gamma C \cap \Theta} \omega_{F}.
	\end{equation}

	Now let $w$ be a coordinate between $x_i$, $y_j$ or $f_k$ and  $g\in G(\mathbb{R})$, denote by $p_{g,w}:g. C\to \mathbb{C}$ the projection to the $w$ axis. Let $w_0\in \mathbb{C}$ and define 
	\begin{equation}
	n_{g.C,w}(w_0)	= \text{ number of points in } g.C\cap p_{w}^{-1}(w_{0}) \text{ counted with multiplicity}.
	\end{equation}
	Consider the set 
	\begin{equation}
		W = \left\{ (z_{0},g,w_{0})\in U_{i}\cap B(0,R_{i})\times G(\mathbb{R})\times\mathbb{C} \mathop{|} p_{g,w}\circ f(z_{0})=w_{0} \right\}.
	\end{equation}
	We observe that, since $\mathcal{D}$ is bounded and $C$ is an analytic subset of $\mathcal{D}$, the projection $p_{g,w}$ is definable in $\mathbb{R}_{an}$. This along with the definablity of all other maps involved, implies that $S$ is a definable family over $G(\mathbb{R})\times\mathbb{C}$. It is a consequence of the cell decomposition theorem (cf. \cite[Chapter 3, Corollary 3.6]{Dries98}) that the number of definably connected components of the fibres of a definable set, hence, in this case, their cardinality, is uniformly bounded by a constant $c_{w}$. We now observe that the fibre of $W$ over a point $(g,w_{0})\in G(\mathbb{R})\times \mathbb{C}$ is the set $f^{-1}(p_{g,w}^{-1}(w_{0}))$ whose cardinality is exactly $n_{g.C,w}(w_{0})$. Hence
	\begin{equation}
	n_{\gamma C,w}(w_{0})\leq c_{w}	
	\end{equation}
	for all $w_{0}\in \mathbb{C}$ and all $\gamma\in \Gamma$. Let $c_{6}$ be the maximum of $c_{w}$ with $w$ equal to $x_{i}$, $y_{j}$ or $f_{k}$, then
	\begin{equation}
		\begin{split}
		\Vol_{\gamma C}(\gamma C \cap \Theta) &\leq c_{5} \left( \vphantom{\int_{p_{x_i}(\Theta)}} \right. \sum  \int_{p_{x_i}(\Theta)} n_{\gamma C }\left(p_{x_{i}}^{-1}(x_{i})\right)\,\frac{dx_i\wedge d \bar{x_{i}}}{\im(x_{i})^2}\ + \\
		&\quad\sum \int_{p_{y_j}(\Theta)} n_{\gamma C }\left(p_{y_{j}}^{-1}(y_{j})\right)\,d y_{j}\wedge d \bar{y_{j}}\ +\\
		&\quad\sum \left. \int_{p_{f_k}(\Theta)} n_{\gamma C }\left(p_{f_{k}}^{-1}(f_{k})\right)\,df_{k}\wedge d \bar{f_k} \right) \\
		&\leq c_{5}c_{6} \left( \sum  \int_{p_{x_i}(\Theta)} \frac{dx_i\wedge d \bar{x_{i}}}{\im(x_{i})^2}+ \sum \int_{p_{y_j}(\Theta)} d y_{j}\wedge d \bar{y_{j}} + \right.\\
		&\left. \hspace{20pt} \int_{p_{f_k}(\Theta)} df_{k}\wedge d \bar{f_k} \right) 
		\end{split}
	\end{equation}
	Now we observe that from the description of $\Theta$, the projection $p_{x_i}(\Theta)$ is contained in a finite union of usual fundamental domains in the upper half plane, which have finite hyperbolic area. Moreover, if $w$ is one of $y_{j}$ or $f_{k}$, then, again from the description of $\Theta$, it follows that $p_{w}(\Theta)$ is relatively compact in the plane and hence has finite Euclidean area.
\end{proof}

This result allows us to follow the proof used in \cite{KlinglerUllmoYafaev16} for algebraic curves and apply it to our o-minimal setting. For the convenience of the reader, we briefly recall the proof of theorem \ref{thm:intersection holomorphic curves and fundamental domains}. First we report some results from \cite{KlinglerUllmoYafaev16}.

\begin{lemma}[{\cite[Lemma 5.4]{KlinglerUllmoYafaev16}}]
	\label{lemma:Lemma 5.4}
	Let $x_{0}\in \mathcal{D}$ be a base point. There exists a constant $c_{7}$ such that for any $g\in G(\mathbb{R})$ the following inequality holds
	\begin{equation}
		\log \left( c_{7}\abs{g}_{\infty} \right) \leq d(g.x_{0},x_{0}).
	\end{equation}
\end{lemma}

\begin{lemma}[{\cite[Lemma 5.5]{KlinglerUllmoYafaev16}}]
	\label{lemma:Lemma 5.5}
	Let $\mathscr{F}$ be the fundamental domain for the action of $\Gamma$ fixed in the previous section. There exists a positive constant $c_{8}$ such that for all $\gamma\in \Gamma$ and for all $u\in \gamma \mathscr{F}$
	\begin{equation}
		H(\gamma)\leq c_{8} \abs{u}_{\infty}^{n}.
	\end{equation}
\end{lemma}

\begin{theorem}[{\cite{HwangTo02}}]
	\label{thm:Theorem 5.7}
	Let $C$ be a complex analytic curve in $\mathcal{D}$. For any point $x_{0}\in C$ there exist positive constants $c_{9}$ and $c_{10}$ such that for any positive real number $R$ one has
	\begin{equation}
		\Vol_{C}(C\cap B(x_{0},R))\geq c_{9}\exp \left( c_{10} R\right) 
	\end{equation}
	where $\Vol_{C}$ is the volume with respect to the riemaniann metric on $C$ induced by the metric on $\mathcal{D}$ and $B(x_{0},R)$ is the geodesic ball in $\mathcal{D}$ of center $x_{0}$ and radius $R$.
\end{theorem}

We can now finish the proof of theorem \ref{thm:intersection holomorphic curves and fundamental domains}.

\begin{proof}[Proof of Theorem \ref{thm:intersection holomorphic curves and fundamental domains}]
	Choose a base point $x_{0}\in C$, let $c_{7}$ and $c_{8}$ the constants given by lemma \ref{lemma:Lemma 5.4} and lemma \ref{lemma:Lemma 5.5} and consider the intersection $C\cap B(x_{0},R)$ of $C$ with the geodesic ball of centre $x_{0}$ and radius $R=\log \left( \frac{c_{7}}{{c_{8}}^{1/n}}T^{1/n} \right) $. On the one hand, we have by theorem \ref{thm:Theorem 5.7}
	\begin{equation}
		\Vol_{C}(C\cap B(x_{0},R))\geq \frac{c_{7}c_{9}}{{c_{8}}^{\sfrac{1}{n}}} T^{\frac{c_{10}}{n}}.
	\end{equation}
	On the other hand, by lemma \ref{lemma:Lemma 5.4} and lemma \ref{lemma:Lemma 5.5}
	\begin{equation}
	\begin{split}
		B(x_{0},\log R)	&\subseteq \left\{ g\cdot x_{0} \mathop{|} g\in G(\mathbb{R}), \abs{g}_{\infty}\leq \sfrac{T^{\sfrac{1}{n}}}{c_{8}^{\sfrac{1}{n}}} \right\}\\
						&\subseteq \bigcup_{\mathclap{\substack{\gamma\in \Gamma\\ H(\gamma)\leq T}}} \gamma \mathcal{F}.
	\end{split}
	\end{equation}
	Hence, by lemma \ref{lemma:lemma 5.8}
	\begin{equation}
		\Vol_{C}(C\cap B(x_{0},\log R))	\leq \sum_{\mathclap{\substack{\gamma\in \Gamma\\ \gamma \mathcal{F}\cap C\neq \emptyset \\ H(\gamma)\leq T}}} \Vol_{\gamma^{-1} C}(\gamma^{-1} C\cap \mathcal{F} )
		\leq N_{C}(T) c_{3}
	\end{equation}
	We conclude comparing the lower bound and the upper bound
	\begin{equation}
		\frac{c_{7}c_{9}}{{c_{8}}^{\sfrac{1}{n}}} T^{\frac{c_{10}}{n}}\leq N_{C}(T) c_{3}
	\end{equation}
\end{proof}


%% file: proof.tex

\section{Proof of Theorem \ref{thm:thm 2.3}} 
\label{sec:proof_of_theorem_ref}

We use the same notation as in section \ref{sec:preliminaries}; that is, we let $f:\mathbb{C}^{m}\to \mathbb{C}^{n}$ be an holomorphic map such that $f(\mathbb{C}^{m})\cap \mathcal{D}\neq \emptyset$. $U_{i}$ are the connected components of $f^{-1}(\mathcal{D})$. Fix an $i$ and set $U=U_{i}$; let $R>0$ be such that $f(U\cap B(0,R))\cap \de X\neq \emptyset$. Let $V=f(U\cap B(0,R))$. Finally, let $W=\Zar(\pi(V))\subset S$.

First of all we note that, by definition, $V$ is definable in the o-minimal structure $\mathbb{R}_{an}$.

\begin{lemma}
	\label{lemma:Sigma}
	Consider the set 
	\begin{equation}
		\Sigma(W)=\left\{ g\in G(\mathbb{R}) \mathop{|} \dim(g.V\cap \mathcal{F}\cap \pi^{-1}(W))= \dim (V) \right\}.
	\end{equation}
	Then the set $\Sigma(W)$ is definable in $\mathbb{R}_{an,exp}$. For all $g\in \Sigma(W)$, $g.V\subseteq \pi^{-1}(W)$. Moreover, define
	\begin{equation}
		\Sigma'(W)=\left\{ g\in G(\mathbb{R}) \mathop{|} V\cap g^{-1}.\mathcal{F}\neq \emptyset \right\}.
	\end{equation}
	Then
	\begin{equation}
		\Sigma(W)\cap \Gamma= \Sigma'(W)\cap \Gamma.
	\end{equation}
\end{lemma}
\begin{proof}
	The set $\Sigma(W)$ is definable in $\mathbb{R}_{an,exp}$ because all sets and maps involved in its definition are\footnote{For the definablility of the uniformisation map see \cite[Section 4]{KlinglerUllmoYafaev16}.}. The second assertion follows by analytic continuation. Finally the equality
	\begin{equation}
		\Sigma(W)\cap \Gamma= \Sigma'(W)\cap \Gamma
	\end{equation}
	follows form the fact that $\pi^{-1}(W)$ is $\Gamma$-invariant.
\end{proof}

We now recall the Pila-Wilkie counting theorem.

\begin{theorem}[{\cite{PilaWilkie06}}]
	\label{thm:Pila-Wilkie}
	Let $S\subset \mathbb{R}^{n}$ be a set definable in the o-minimal structure $\mathbb{R}_{an,exp}$. Let $H: \mathbb{Q}^{n}\to \mathbb{R}$ be the usual multiplicative height and denote by $N_{S}(T)$ the number of rational points in $S$ of height at most $T$. If there exist constants $c,\varepsilon>0$ such that $N_{S}(T)>cT^{\varepsilon}$ for all $T$ sufficiently large, then $S$ contains a positive dimensional semialgebraic subset.
\end{theorem}

Let 
	\begin{equation}
		N_{\Sigma(W)}(T)= \left\{ \gamma\in \Gamma \cap \Sigma(W) \mathop{|} H(\gamma)\leq T \right\}.
	\end{equation}
From theorem \ref{thm:intersection holomorphic curves and fundamental domains}, we see that $N_{\Sigma(W)}(T)\geq c_{1}T^{c_{2}}$, for some constants $c_{1},c_{2}>0$. Combining this with theorem \ref{thm:Pila-Wilkie}, we get a semialgebraic set $X\subset \Sigma(W)$. Finally from lemma \ref{lemma:Sigma} we get that $X . V\subset \pi^{-1}(W)$, thus proving theorem \ref{thm:thm 2.3}.
